\documentclass{amsart}

\usepackage{dsfont,amssymb,tikz-cd,hyperref}

\newtheorem*{maintheorem}{Main Theorem}

\newtheorem{theorem}{Theorem}[section]
\newtheorem{corollary}[theorem]{Corollary}
\newtheorem{lemma}[theorem]{Lemma}
\newtheorem{proposition}[theorem]{Proposition}

\theoremstyle{definition}

\newtheorem{example}[theorem]{Example}
\newtheorem{remark}[theorem]{Remark}

\numberwithin{equation}{section}

\DeclareMathOperator{\add}{add}
\DeclareMathOperator{\fp}{fp}
\DeclareMathOperator{\Hom}{Hom}
\let\Im\relax\DeclareMathOperator{\Im}{Im}
\DeclareMathOperator{\inj}{inj}
\DeclareMathOperator{\Ker}{Ker}
\DeclareMathOperator{\proj}{proj}
\DeclareMathOperator{\rad}{rad}

\DeclareMathOperator{\Rep}{Rep}
\DeclareMathOperator{\soc}{soc}
\DeclareMathOperator{\supp}{supp}

\newcommand{\1}{\mathds{1}}

\newcommand{\To}{\longrightarrow}

\newcommand{\abs}[1]{\left\lvert#1\right\rvert}
\newcommand{\set}[1]{\left\{#1\right\}}

\title[Injective objects in $\mathrm{fp}(Q)$]
  {Injective objects in the category of finitely
  presented representations of an interval finite quiver}

\date{\today}

\author{Pengjie Jiao}
\address{Department of Mathematics,
  China Jiliang University,
  Hangzhou 310018, PR China}
\email{jiaopjie@cjlu.edu.cn}

\begin{document}

\begin{abstract}
  We characterize the indecomposable injective objects in the category of finitely presented representations of an interval finite quiver.
\end{abstract}

\maketitle

\section{Introduction}

Infinite quivers appear naturally in the covering theory of algebras;
see such as \cite{BongartzGabriel1982Covering,Gabriel1981universal}.
The injective representations of an infinite quiver $Q$ over an arbitrary ring $R$ is studied in \cite{EnochsEstradaGarciaRozas2009Injective}.
We are interested in the category $\fp(Q)$ of finitely presented representations when $R$ is a field.

Recall that $\fp(Q)$ is studied in \cite{ReitenVandenBergh2002Noetherian} when $Q$ is locally finite of certain type.
The result is used to classify the Noetherian Ext-finite hereditary abelian categories with Serre duality.
More generally, when $Q$ is strongly locally finite, the Auslander--Reiten quiver of $\fp(Q)$ is studied in \cite{BautistaLiuPaquette2013Representation}.
The result is used to study the bounded derived category of a finite dimensional algebra with radical square zero in \cite{BautistaLiu2017bounded}.

In the study of Auslander--Reiten theory of $\fp(Q)$, a natural question is how about the injective objects.
We find that we can deal with it when $Q$ is \emph{interval finite}
(i.e., for any vertices $a$ and $b$, the set of paths from $a$ to $b$ is finite).

For each vertex $a$, we denote by $I_a$ the corresponding indecomposable injective representation.
Let $p$ be a left infinite path, i.e., an infinite sequence of arrows $\cdots \alpha_i \cdots \alpha_2 \alpha_1$ with $s(\alpha_{i+1}) = t(\alpha_i)$ for any $i \geq 1$.
Denote by $[p]$ the equivalence class (see page~\pageref{def:equiv} for the definition) of left infinite paths containing $p$.
Consider the indecomposable representation $Y_{[p]}$ introduced in \cite[Section~5]{Jiao2019Projective}.
We have that if $Y_{[p]}$ lies in $\fp(Q)$, then it is an indecomposable injective object; see Proposition~\ref{prop:Yp-inj}.

Moreover, we can classify the indecomposable injective objects in $\fp(Q)$.

\begin{maintheorem}[see~Theorem~\ref{thm:classify}]
  Let $Q$ be an interval finite quiver. Assume $I$ is an indecomposable injective object in $\fp(Q)$. Then either $I \simeq I_a$ for certain vertex $a$, or $I \simeq Y_{[p]}$ for certain left infinite path $p$.
\end{maintheorem}

Compared with \cite[Theorem~6.8]{Jiao2019Projective}, the difficulty here is to characterize when $I_a$ and $Y_{[p]}$ are finitely presented.
The result strengthens a description of finite dimensional indecomposable injective objects in $\fp(Q)$; see \cite[Proposition~1.16]{BautistaLiuPaquette2013Representation}.

The paper is organized as follows.
In Section~2, we recall some basic facts about quivers and representations.
In Section~3, we study the injective objects in $\fp(Q)$ and give the classification theorem.
Some examples are given in Section~4.

\section{Quivers and representations}

Let $k$ be a field, and $Q=(Q_0,Q_1)$ be a quiver, where $Q_0$ is the set of vertices and $Q_1$ is the set of arrows.
For each arrow $\alpha \colon a \to b$, we denote by $s(\alpha) = a$ its source and by $t(\alpha) = b$ its target.

A path $p$ of length $l\geq1$ is a sequence of arrows $\alpha_l \cdots \alpha_2 \alpha_1$ such that $s(\alpha_{i+1}) = t(\alpha_i)$ for any $1 \leq i \leq l-1$. We set $s(p) = s(\alpha_1)$ and $t(p) = t(\alpha_l)$.
We associate each vertex $a$ with a trivial path (of length 0) $e_a$ with $s(e_a) = a = t(e_a)$.
A nontrivial path $p$ is called an oriented cycle if $s(p) = t(p)$.
For any $a, b \in Q_0$, we denote by $Q(a,b)$ the set of paths $p$ from $a$ to $b$, i.e., $s(p) = a$ and $t(p) = b$.

If $Q(a, b) \neq \emptyset$, then $a$ is called a predecessor of $b$, and $b$ is called a successor of $a$.
For $a \in Q_0$, we denote by $a^-$ the set of vertices $b$ with some arrow $b \to a$; by $a^+$ the set of vertices $b$ with some arrow $a \to b$.

A \emph{right infinite path} $p$ is an infinite sequence of arrows $\alpha_1 \alpha_2 \cdots \alpha_n \cdots$ such that $s(\alpha_i) = t(\alpha_{i+1})$ for any $i \geq 1$. We set $t(p) = t(\alpha_1)$.
Dually, a \emph{left infinite path} $p$ is an infinite sequence of arrows $\cdots \alpha_n \cdots \alpha_2 \alpha_1$ such that $s(\alpha_{i+1}) = t(\alpha_i)$ for any $i \geq 1$. We set $s(p) = s(\alpha_1)$.
Here, we use the terminologies in \cite[Section~2.1]{Chen2015Irreducible}.
We mention that these are opposite to the corresponding notions in
\cite[Section~1]{BautistaLiuPaquette2013Representation}.

A representation $M = ( M(a), M(\alpha) )$ of $Q$ over $k$ means a collection of $k$-linear spaces $M(a)$ for every $a \in Q_0$, and a collection of $k$-linear maps $M(\alpha) \colon M(a) \to M(b)$ for every arrow $\alpha \colon a \to b$.
For each nontrivial path $p = \alpha_l \cdots \alpha_2 \alpha_1$, we denote
\(
  M(p) = M(\alpha_l) \circ \cdots \circ M(\alpha_2) \circ M(\alpha_1)
\).
For each $a \in Q_0$, we set $M(e_a) = \1_{M(a)}$.
A morphism $f \colon M \to N$ of representations is a collection of $k$-linear maps $f_a \colon M(a) \to N(a)$ for every $a \in Q_0$, such that $f_b \circ M(\alpha) = N(\alpha) \circ f_a$ for any arrow $\alpha \colon a \to b$.

Let $\Rep(Q)$ be the category of representation of $Q$ over $k$. We denote by $\Hom(M, N)$ the set of morphisms from $M$ to $N$ in $\Rep(Q)$. It is well known that $\Rep(Q)$ is a hereditary abelian category;
see \cite[Section~8.2]{GabrielRoiter1992Representations}.

Recall that a subquiver $Q'$ of $Q$ is called \emph{full} if any arrow $\alpha$ with $s(\alpha), t(\alpha) \in Q'_0$ lies in $Q'$.
Let $M$ be a representation of $Q$.
The \emph{support} $\supp M$ of $M$ is the full subquiver of $Q$ formed by vertices $a$ with $M(a) \neq 0$.
The \emph{socle} $\soc M$ of $M$ is the subrepresentation such that
\(
  (\soc M) (a) = \bigcap_{\alpha \in Q_1, s(\alpha) = a} \Ker M(\alpha)
\)
for any vertex $a$.
The \emph{radical} $\rad M$ of $M$ is the subrepresentation such that
\(
  (\rad M) (a) = \sum_{\alpha \in Q_1, t(\alpha) = a} \Im M(\alpha)
\)
for any vertex $a$.

%

We mention the following fact;
see \cite[Lemma~1.1]{BautistaLiuPaquette2013Representation}.

\begin{lemma}\label{lem:soc-rad}
  If the support of a representation $M$ contains no left infinite paths, then $\soc M$ is essential in $M$.
\end{lemma}

\begin{proof}
  Let $N$ be a nonzero subrepresentation of $M$. Assume $x \in N(a)$ is nonzero for some vertex $a$. Since $\supp M$ contains no left infinite paths, there exists some path $p$ in $\supp M$ with $s(p) = a$ such that $N(p)(x) \neq 0$ and $N(\alpha p)(x) = 0$ for any arrow $\alpha$ in $Q$. Then $N(p)(x) \in ( N \cap \soc M )(t(p))$. It follows that $\soc M$ is essential in $M$.
\end{proof}

Let $a$ be a vertex in $Q$. We define a representation $P_a$ as follows.
For every vertex $b$, we let
\[
  P_a (b) = \bigoplus_{p \in Q(a,b)} k p.
\]
For every arrow $\alpha \colon b \to b'$, we let
\[
  P_a (\alpha) \colon P_a (b) \To P_a (b'),
  \enskip p \mapsto \alpha p.
\]
Similarly, we define a representation $I_a$ as follows.
For every vertex $b$, we let
\[
  I_a (b)= \Hom_k \biggl( \bigoplus_{p \in Q(b,a)} k p, k \biggr).
\]
For every arrow $\alpha \colon b \to b'$, we let
\[
  I_a (\alpha) \colon I_a (b) \To I_a (b'),
  \enskip f \mapsto ( p \mapsto f( p \alpha ) ).
\]

The following result is well known;
see \cite[Section~3.7]{GabrielRoiter1992Representations}.
It implies that $P_a$ is a projective representation and $I_a$ is an injective representation in $\Rep(Q)$.

\begin{lemma}\label{lem:proj-inj}
  Let $M \in \Rep(Q)$ and $a \in Q_0$.
  \begin{enumerate}
    \item
      The $k$-linear map
      \[
        \eta_M \colon \Hom(P_a, M) \To M(a),
        \enskip f \mapsto f_a (e_a),
      \]
      is an isomorphism natural in $M$.
    \item
      The $k$-linear map
      \[
        \zeta_M \colon \Hom(M, I_a) \To \Hom_k (M(a), k),
        \enskip f \mapsto (x \mapsto f_a (x) (e_a)),
      \]
      is an isomorphism natural in $M$.
  \end{enumerate}
\end{lemma}

\begin{proof}
  (1) Consider the $k$-linear map
  \[
    \eta'_M \colon M(a) \To \Hom(P_a, M)
  \]
  given by
  \(
    ( \eta'_M (x) )_b (p) = M(p) (x)
  \)
  for any $x \in M(a)$, $b \in Q_0$ and $p \in Q(a, b)$.
  We observe that
  \(
    \eta'_M \circ \eta_M = \1_{\Hom(P_a, M)}
  \)
  and
  \(
    \eta_M \circ \eta'_M = \1_{M(a)}
  \).
  Then $\eta_M$ is an isomorphism.

  (2) Consider the $k$-linear map
  \[
    \zeta'_M \colon \Hom_k (M(a), k) \To \Hom(M, I_a)
  \]
  given by
  \(
    ( \zeta'_M (f) )_b (x) (p) = f ( M(p)(x) )
  \)
  for any $f \in \Hom_k (M(a), k)$, $b \in Q_0$, $x \in M(b)$ and $p \in Q(b, a)$.
  We observe that
  \(
    \zeta'_M \circ \zeta_M = \1_{\Hom(M, I_a)}
  \)
  and
  \(
    \zeta_M \circ \zeta'_M = \1_{\Hom_k (M(a), k)}
  \).
  It follows that $\zeta_M$ is an isomorphism.
\end{proof}

An epimorphism $P \to M$ with projective $P$ is called a projective cover of $M$ if it is an essential epimorphism. A monomorphism $M \to I$ with injective $I$ is called an injective envelope of $M$ if it is an essential monomorphism.
We mention that two injective envelopes of $M$ are isomorphic.

Given a collection $\mathcal{A}$ of representations, we denote by $\add \mathcal{A}$ the full subcategory of $\Rep(Q)$ formed by direct summands of finite direct sums of representations in $\mathcal{A}$.
We set
\(
  \proj(Q) = \add \set{P_a | a \in Q_0}
\)
and
\(
  \inj(Q) = \add \set{I_a | a \in Q_0}
\).

A representation $M$ is called \emph{finitely generated} if there exists some epimorphism
\(
  f \colon \bigoplus_{i=1}^n P_{a_i} \to M
\),
and is called \emph{finitely presented} if moreover $\Ker f$ is also finitely generated.
We denote by $\fp(Q)$ the subcategory of $\Rep(Q)$ formed by finitely presented representations.

We have the following well-known fact.

\begin{proposition}\label{prop:fp-abel}
  The category $\fp(Q)$ is a hereditary abelian subcategory of $\Rep(Q)$ closed under extensions.
\end{proposition}

\begin{proof}
  Let $f \colon P \to P'$ be a morphism in $\proj(Q)$.
  We observe that $\Im f$ is projective since $\Rep(Q)$ is hereditary.
  Then the induced exact sequence
  \[
    0 \To \Ker f \To P \To \Im f \To 0
  \]
  splits. Therefore $\Ker f \in \proj(Q)$.
  It follows from \cite[Proposition~2.1]{Auslander1966Coherent} that $\fp(Q)$ is abelian. We observe by the horseshoe lemma that $\fp(Q)$ is closed under extensions in $\Rep(Q)$. In particular, it is hereditary.
\end{proof}

We mention the following observation.

\begin{lemma}\label{lem:fp-M/N}
  Let $M$ be a finitely presented representation and $N$ be a finitely generated subrepresentation. Then $M/N$ is finitely presented.
\end{lemma}

\begin{proof}
  Let $f \colon P \to M$ be an epimorphism with $P \in \proj(Q)$. Then $\Ker f$ is finitely generated. Denote by $g$ the composition of $f$ and the canonical surjection $M \to M/N$. Consider the following commutative diagram.
  \[\begin{tikzcd}
    0 \rar & \Ker g \rar & P \rar["g"] & M/N \rar & 0 \\
    0 \rar & N      \rar & M \rar      & M/N \rar & 0
    \ar[from = 1-2, to = 2-2, dashed, "h"']
    \ar[from = 1-3, to = 2-3, "f"']
    \ar[from = 1-4, to = 2-4, equal]
  \end{tikzcd}\]
  We observe that the left square is a pushout and also a pullback. Then $h$ is an epimorphism and $\Ker h \simeq \Ker f$. In particular, $\Ker h$ is finitely generated.
  Consider the exact sequence
  \[
    0 \To \Ker h \To \Ker g \overset{h}\To N \To 0.
  \]
  It follows that $\Ker g$ is finitely generated. Then $M/N$ is finitely presented.
\end{proof}

The injective objects in $\fp(Q)$ satisfy the following property.

\begin{lemma}\label{lem:alpha-epi}
  Let $I$ be an injective object in $\fp(Q)$, and let $a \in Q_0$. Assume $p_i$ is a path from $a$ to $b_i$ for $1 \leq i \leq n$ such that $p_i$ is not of the form $u p_j$ with $u \in Q(b_j, b_i)$ for any $j \neq i$.
  Then the $k$-linear map
  \[
    \begin{pmatrix}
      I(p_1) \\
      \vdots \\
      I(p_n) \\
    \end{pmatrix}
    \colon I(a) \To \bigoplus_{i=1}^n I(b_i)
  \]
  is a surjection.
\end{lemma}

We mention that one can consider the special case that $p_i \colon a \to b_i$ for $1 \leq i \leq n$ are pairwise different arrows.

\begin{proof}
  We observe that the canonical morphism $\bigoplus_{i=1}^n P_{b_i} \to P_a$ induced by inclusions is a monomorphism. The injectivity of $I$ gives a surjection $\Hom(P_a, I) \to \Hom(\bigoplus_{i=1}^n P_{b_i}, I)$. By identifying $\Hom(P_c, I)$ and $I(c)$ for any vertex $c$, we observe that the surjection is precisely the map needed.
\end{proof}

The following fact is a direct consequence.

\begin{corollary}\label{cor:inj-supp-a-}
  The support of an injective object in $\fp(Q)$ is closed under predecessors.
\end{corollary}

\begin{proof}
  Let $I$ be an injective object in $\fp(Q)$. Assume $a$ is a vertex in $\supp I$ and $b \in a^-$. We can choose some arrow $\alpha \colon b \to a$. Lemma~\ref{lem:alpha-epi} implies that $I(\alpha)$ is a surjection. In particular, the vertex $b$ lies in $\supp I$. Then the result follows.
\end{proof}

Recall that $Q$ is called \emph{interval finite} if $Q(a,b)$ is finite for any $a,b \in Q_0$.
A quiver is called \emph{top finite} if there exist finitely many vertices of which every vertex is a successor, and is called \emph{socle finite} if there exist finitely many vertices of which every vertex is a predecessor.

We have the following observation.

\begin{lemma}\label{lem:inf-path}
  A top finite interval finite quiver contains no right infinite paths; a socle finite interval finite quiver contains no left infinite paths.
\end{lemma}

\begin{proof}
  Let $Q$ be a top finite interval finite quiver. Then there exist some vertices $b_1, b_2, \dots, b_n$ such that any vertex is a successor of some $b_i$. Assume $Q$ contains a right infinite path $\alpha_1 \alpha_2 \cdots \alpha_j \cdots$. For each $j \geq 0$, we set $a_j = t(\alpha_{j+1})$.
  Since $Q(b_i, a_0)$ is finite, there exists some nonnegative integer $Z_i$ such that $Q(b_i, a_j) = \emptyset$ for any $j \geq Z_i$. Let $Z = \max_{1 \leq i \leq n} Z_i$. Then $a_Z$ is not a successor of any $b_i$, which is a contradiction. It follows that $Q$ contains no right infinite paths.

  Similarly, a socle finite interval finite quiver contains no left infinite paths.
\end{proof}

\section{Finitely presented representations}

Let $k$ be a field and $Q$ be an interval finite quiver.

Recall that a representation $M$ is called \emph{pointwise finite dimensional} if $M(a)$ is finite dimensional for any vertex $a$, and is called \emph{finite dimensional} if moreover $\supp M$ contains only finitely many vertices.

We mention the following fact.

\begin{lemma}\label{lem:fp-pf}
  The abelian category $\fp(Q)$ is Hom-finite Krull--Schmidt, and every object is pointwise finite dimensional.
\end{lemma}

\begin{proof}
  Assume $\bigoplus_{i=1}^m P_{a_i} \to M$ is an epimorphism. We observe that each $P_{a_i}$ is pointwise finite dimensional, since $Q$ is interval finite. Then so is $M$.

  Moreover, assume $\bigoplus_{j=1}^n P_{b_j} \to N$ is an epimorphism.
  Consider the maps
  \[
    \Hom(M, N)
    \hookrightarrow
    \Hom\biggl(\bigoplus_{i=1}^m P_{a_i}, N\biggr)
    \twoheadleftarrow
    \Hom\biggl(\bigoplus_{i=1}^m P_{a_i}, \bigoplus_{j=1}^n P_{b_j}\biggr).
  \]
  We observe that
  \(
    \Hom(\bigoplus_{i=1}^m P_{a_i}, \bigoplus_{j=1}^n P_{b_j})
  \)
  is finite dimensional since $Q$ is interval finite. Then so is $\Hom(M, N)$.
  Therefore the abelian category $\fp(Q)$ is Hom-finite, and hence is Krull--Schmidt.
\end{proof}

We need the following properties of finitely presented representations.

\begin{lemma}\label{lem:fp-supp}
  Let $M$ be a finitely presented representation.
  \begin{enumerate}
    \item
      $\supp M$ is top finite.
    \item
      \(
        \bigcup_{a \in \supp M} a^+ \setminus \supp M
      \)
      is finite.
  \end{enumerate}
\end{lemma}

\begin{proof}
  (1) Assume
  \(
    f \colon \bigoplus_{i=1}^m P_{a_i} \to M
  \)
  is an epimorphism.
  Then every vertex in $\supp M$ is a successor of some $a_i$. In other words, $\supp M$ is top finite.

  (2) Denote
  \(
    \Delta = \bigcup_{a \in \supp M} a^+ \setminus \supp M
  \).
  We observe that $\Ker f / \rad \Ker f$ is semisimple and
  \(
    ( \Ker f / \rad \Ker f ) (b) \neq 0
  \)
  for any $b \in \Delta$.
  If $\Delta$ is not finite, then $\Ker f / \rad \Ker f$ is not finitely generated, which is a contradiction.
  It follows that $\Delta$ is finite.
\end{proof}

\begin{corollary}\label{cor:fd-fp}
  A finite dimensional representation $M$ is finitely presented if and only if $a^+$ is finite for any vertex $a$ in $\supp M$.
\end{corollary}

\begin{proof}
  For the necessary, we assume $M$ is finitely presented and $a$ is a vertex in $\supp M$.
  Lemma~\ref{lem:fp-supp} implies that $a^+ \setminus \supp M$ is finite. Since $\supp M$ contains only finitely many vertices, then $a^+ \cap \supp M$ is finite. It follows that $a^+$ is finite.

  For the sufficiency, we assume $a^+$ is finite for any vertex $a$ in $\supp M$.
  Since $M$ is finite dimensional, there exists some epimorphism $f \colon P \to M$ with $P \in \proj(Q)$.

  Consider the subrepresentation $N$ of $P$ generated by $P(b)$, where $b$ runs over
  \(
    \bigcup_{a \in \supp M} a^+ \setminus \supp M
  \).
  Then $N$ is contained in $\Ker f$.
  We observe by Lemma~\ref{lem:fp-pf} each $P(b)$ is finite dimensional. Then $N$ is finitely generated.

  Consider the factor module $\Ker f / N$. Its support is contained in $\supp M$. Then it is finite dimensional and hence is finitely generated.
  Consider the exact sequence
  \[
    0 \To N \To \Ker f \To \Ker f / N \To 0.
  \]
  It follows that $\Ker f$ is finitely generated, and then $M$ is finitely presented.
\end{proof}

\begin{corollary}\label{cor:Ia-fp}
  Let $a$ be a vertex. Then $I_a$ is finitely presented if and only if
  $a$ admits only finitely many predecessors $b$ and each $b^+$ is finite.
\end{corollary}

\begin{proof}
  We observe that $I_a$ is finitely generated if and only if $\supp I$ contains only finitely many vertices, since $Q$ is interval finite.
  The vertices in $\supp I$ are precisely predecessors of $a$.
  Then the result follows from Corollary~\ref{cor:fd-fp}.
\end{proof}

The support of an injective object in $\fp(Q)$ satisfies the following conditions.

\begin{lemma}\label{lem:inj-supp-f}
  Let $I$ be an injective object in $\fp(Q)$.
  \begin{enumerate}
    \item
      $a^- \cup a^+$ is finite for any vertex $a$ in $\supp I$.
    \item
      If $\supp I$ contains no left infinite paths, then it contains only finitely many vertices.
  \end{enumerate}
\end{lemma}

\begin{proof}
  Lemma~\ref{lem:fp-supp} implies that $\supp I$ is top finite. Then there exist some vertices $b_1, b_2, \dots, b_n$ such that any vertex in $\supp I$ is a successor of some $b_i$.

  (1) We observe that $\supp I$ is closed under predecessors; see Corollary~\ref{cor:inj-supp-a-}. Then any vertex in $a^-$ is a successor of some $b_i$.
  If $a^-$ is infinite, then at least one $Q(b_i, a)$ is infinite, which is a contradiction. It follows that $a^-$ is finite.

  We observe that $a^+ \cap \supp I$ is finite.
  Indeed, otherwise Lemma~\ref{lem:alpha-epi} implies that $I(a)$ is not finite dimensional, which is a contradiction. Since $a^+ \setminus \supp I$ is finite by Lemma~\ref{lem:fp-supp}, then $a^+$ is finite. It follows that $a^- \cup a^+$ is finite.

  (2) Assume the vertices in $\supp I$ is infinite. Then there exists some $b_i$ whose successors contained in $\supp I$ is infinite. Denote it by $a_0$. Since $a_0^+ \cap \supp I$ is finite, then there exists some $a_1 \in a_0^+ \cap \supp I$ whose successors contained in $\supp I$ is infinite. Choose some arrow $\alpha_1 \colon a_0 \to a_1$.

  By induction, we obtain vertices $a_i$ and arrows $\alpha_{i+1} \colon a_i \to a_{i+1}$ for $i \geq 0$ in $\supp M$. This is a contradiction, since $\cdots \alpha_i \cdots \alpha_2 \alpha_1$ is a left infinite path in $\supp M$.
  It follows that $\supp M$ contains only finitely many vertices.
\end{proof}

For an injective object in $\fp(Q)$ whose support contains no left infinite paths, we have the following characterization.

\begin{proposition}\label{prop:inj(Q)}
  Let $I$ be an injective object in $\fp(Q)$ such that $\supp I$ contains no left infinite paths.
  Then
  \[
    I \simeq \bigoplus_{a \in Q_0} I_a^{\oplus \dim (\soc I) (a)}.
  \]
\end{proposition}

\begin{proof}
  It follows from Lemma~\ref{lem:inj-supp-f} that $\supp I$ contains only finitely many vertices.
  Let
  \(
    J = \bigoplus_{a \in Q_0} I_a^{\oplus \dim (\soc I) (a)}
  \).
  This is a finite direct sum, since the vertices in $\supp I$ are finite and $I$ is pointwise finite dimensional.

  For any vertex $a$ in $\supp I$, its predecessors are also contained in $\supp I$;
  see Corollary~\ref{cor:inj-supp-a-}.
  It follows that $\supp J$ is a subquiver of $\supp I$.
  Then $\soc I$ and $J$ are finite dimensional.
  Corollary~\ref{cor:fd-fp} implies that they are finitely presented.
  We observe by Lemma~\ref{lem:soc-rad} that the inclusion $\soc I \subseteq I$ and the injection $\soc I \to J$ are injective envelopes in $\fp(Q)$.
  Then the result follows.
\end{proof}

For an injective object in $\fp(Q)$ whose support contains some left infinite paths, we mention the following facts.
They will be used technically in the proof of Theorem~\ref{thm:classify}.

\begin{lemma}\label{lem:inj-supp-p-finite}
  Let $I$ be an injective object in $\fp(Q)$, whose support contains some left infinite paths.
  Denote by $\Delta$ the set of left infinite paths $p$ contained in $\supp I$ with $s(p)^- \cap \supp I = \emptyset$.
  Then $\Delta$ is finite, and every left infinite path $p$ in $\supp I$ admits some path $u$ with $p u \in \Delta$.
\end{lemma}

\begin{proof}
  Lemma~\ref{lem:fp-supp} implies that $\supp I$ is top finite. Assume vertices $b_1, b_2, \dots, b_n$ satisfy that any vertex in $\supp I$ is a successor of some $b_i$. We can assume each $b_i^- \cap \supp I = \emptyset$.

  Let $p$ be a left infinite path in $\supp I$. We observe that $s(p)$ is a successor of some $b_i$. Choose some $u \in Q(b_i, s(p))$. Then $p u \in \Delta$.
  In particular, if $p \in \Delta$ then $b_i = s(p)$ since $s(p)^- \cap \supp I = \emptyset$.

  Assume $\Delta$ is infinite. Then for
  \(
    Z = \max_{1 \leq i \leq n} \dim I(b_i)
  \),
  One can find $n Z + 1$ paths $u_j$ from some $b_i$ such that each $u_j$ is not of the form $v u_{j'}$ for any $j' \neq j$.
  We observe that at least one $1 \leq i \leq n$ such that the number of $u_j$ from $b_i$ is greater than $Z$.
  Then Lemma~\ref{lem:alpha-epi} implies that $\dim I(b_i) > Z$, which is a contradiction.
  Then the result follows.
\end{proof}

\begin{lemma}\label{lem:inj-supp-p}
  Let $I$ be an injective object in $\fp(Q)$, whose support contains some left infinite path $\cdots \alpha_i \cdots \alpha_2 \alpha_1$. Set $a_i = s(\alpha_{i+1})$ for any $i \geq 0$. Then there exists some nonnegative integer $Z$ such that $a_i^+ = \set{a_{i+1}}$, $a_{i+1}^- = \set{a_i}$ and
  $I(\alpha_{i+1})$ is a bijection for any $i \geq Z$.
\end{lemma}

\begin{proof}
  We observe by Lemma~\ref{lem:fp-supp} that $\supp I$ is top finite and there exists some nonnegative integer $Z_1$ such that $a_i^+$ is contained in $\supp I$ for any $i \geq Z_1$.

  It follows from Lemma~\ref{lem:alpha-epi} that $\dim I(a_i) \geq \dim I(a_{i+1})$ for any $i \geq 0$. Then there exists some nonnegative integer $Z_2 \geq Z_1$ such that $\dim I(a_i) = \dim I(a_{Z_2})$ for any $i \geq Z_2$.
  Since $I(\alpha_{i+1})$ is a surjection by Lemma~\ref{lem:alpha-epi}, then it is a bijection.

  We claim that $a_i^+ = \set{a_{i+1}}$ and $Q(a_i, a_{i+1}) = \set{\alpha_{i+1}}$ for any $i \geq Z_2$. Indeed, otherwise there exist some arrow $\beta \colon a_i \to b$ in $\supp I$ with $i \geq Z_2$ and $\beta \neq \alpha_{i+1}$.
  Then Lemma~\ref{lem:alpha-epi} implies that
  \(
    \dim I(a_i) \geq \dim I(a_{i+1}) + \dim I(b) > \dim I(a_{i+1})
  \),
  which is a contradiction.

  Assume vertices $b_1, b_2, \dots, b_n$ satisfy that every vertex in $\supp I$ is a successor of some $b_j$.
  We observe that
  \(
    \abs{Q(b_j, a_{i+1})} \geq \abs{Q(b_j, a_i)}
  \)
  for any $1 \leq j \leq n$ and $i \geq 0$.
  If moreover $\abs{a_{i+1}^-} > 1$ for some $i \geq 0$, then there exists some $j$ such that
  \(
    \abs{Q(b_j, a_{i+1})} > \abs{Q(b_j, a_i)}
  \).

  We claim the existence of nonnegative integer $Z_3$ such that $a_{i+1}^- = \set{a_i}$ for any $i \geq Z_3$.
  Indeed, otherwise there exists some $1 \leq j \leq n$ such that $\set{\abs{Q(b_j, a_i)} | i \geq 0}$ is unbounded. Then Lemma~\ref{lem:alpha-epi} implies that $I(b_j)$ is not finite dimensional, which is a contradiction.

  Let $Z = \max \set{Z_2, Z_3}$. Then the result follows.
\end{proof}

Following \cite[Subsection~2.1]{Chen2015Irreducible}, we define an equivalence relation on left infinite paths.
\label{def:equiv}
Two left infinite paths $\cdots \alpha_i \cdots \alpha_2 \alpha_1$ and $\cdots \beta_i \cdots \beta_2 \beta_1$ are equivalent if there exist some positive integers $m$ and $n$ such that
\[
  \cdots \alpha_i \cdots \alpha_{m+1} \alpha_m
  = \cdots \beta_i \cdots \beta_{n+1} \beta_n.
\]

Let $p$ be a left infinite path.
We denote by $[p]$ the equivalence class containing $p$. We mention that $[p]$ is a set. For any vertex $a$, we denote by $[p]_a$ the subset of $[p]$ formed by left infinite paths $u$ with $s(u) = a$.

Considering \cite[Section~5]{Jiao2019Projective} and \cite[Subsection~3.1]{Chen2015Irreducible}, we introduce a representation $Y_{[p]}$ as follows.
For every vertex $a$, we let
\[
  Y_{[p]} (a) = \Hom_k \biggl( \bigoplus_{u \in [p]_a} k u, k \biggr).
\]
For every arrow $\alpha \colon a \to b$, we let
\[
  Y_{[p]} (\alpha) \colon Y_{[p]} (a) \To Y_{[p]} (b),
  \enskip f \mapsto ( u \mapsto f( u \alpha ) ).
\]
We mention that these $Y_{[p]}$ are indecomposable and pairwise non-isomorphic;
see the dual of \cite[Proposition~5.4]{Jiao2019Projective}.

Recall that a quiver is called \emph{uniformly interval finite}, if there exists some positive integer $Z$ such that for any vertices $a$ and $b$, the number of paths $p$ from $a$ to $b$ is less than or equal to $Z$;
see \cite[Definition~2.3]{Jiao2019Projective}.

We characterize when $Y_{[p]}$ is finitely presented.

\begin{lemma}\label{lem:Yp-fp}
  Let $p$ be a left infinite path. Then $Y_{[p]}$ is finitely presented if and only if
  $\supp Y_{[p]}$ is top finite uniformly interval finite and
  \(
    \bigcup_{a \in \supp Y_{[p]}} a^+ \setminus \supp Y_{[p]}
  \)
  is finite.
\end{lemma}

\begin{proof}
  For the necessary, we assume $Y_{[p]}$ is finitely presented.
  It follows from Lemma~\ref{lem:fp-supp} that $\supp Y_{[p]}$ is top finite and
  \(
    \bigcup_{a \in \supp Y_{[p]}} a^+ \setminus \supp Y_{[p]}
  \)
  is finite.

  Assume vertices $b_1, b_2, \dots, b_n$ satisfy that any vertex in $\supp Y_{[p]}$ is a successor of some $b_i$.
  Let $a$ and $a'$ be a pair of vertices in $\supp Y_{[p]}$. Then there exists some $Q(b_i, a) \neq \emptyset$.
  Set
  \(
    Z = \max_{1 \leq i \leq n} \dim Y_{[p]} (b_i)
  \).
  Since $Q$ contains no oriented cycles, we have that
  \[
    \abs{Q(a, a')} \leq \abs{Q(b_i, a')} \leq \abs{[p]_{b_i}} = \dim Y_{[p]} (b_i) \leq Z.
  \]
  It follows that $\supp Y_{[p]}$ is uniformly interval finite.

  For the sufficiency, we assume
  \(
    p = \cdots \alpha_i \cdots \alpha_2 \alpha_1
  \).
  Set $a_i = s(\alpha_{i+1})$ for any $i \geq 0$.
  Assume vertices $b_1, b_2, \dots, b_n$ satisfy that any vertex in $\supp Y_{[p]}$ is a successor of some $b_i$.
  Since $\supp Y_{[p]}$ is uniformly interval finite, then
  \[
    \set{\abs{Q(b_j, a_i)} | i \geq 0, 1 \leq j \leq n}
  \]
  is bounded. We observe that $\abs{Q(b_j, a_i)} \leq \abs{Q(b_j, a_{i+1})}$.
  Then there exists some nonnegative integer $Z_1$ such that $\abs{Q(b_j, a_i)} = \abs{Q(b_j, a_{Z_1})}$ for any $i \geq Z_1$ and $1 \leq j \leq n$.
  In particular, $a_{i+1}^- = \set{a_i}$ and $Q(a_i, a_{i+1}) = \set{\alpha_{i+1}}$ for any $i \geq Z_1$.

  Since
  \(
    \bigcup_{a \in \supp Y_{[p]}} a^+ \setminus \supp Y_{[p]}
  \)
  is finite, there exists some nonnegative integer $Z_2$ such that $a_i^+$ is contained in $\supp Y_{[p]}$ for any $i \geq Z_2$.
  Let $Z = \max \set{Z_1, Z_2}$. Then $a_{i+1}^- = \set{a_i}$, $a_i^+ = \set{a_{i+1}}$ and $Y_{[p]} (\alpha_{i+1})$ is a bijection.

  Consider the subrepresentation $N$ of $Y_{[p]}$ generated by $Y_{[p]}(a_Z)$.
  We observe that
  \(
    N \simeq P_{a_Z}^{\oplus \dim Y_{[p]}(a_Z)}
  \)
  and hence is finitely presented.
  Moreover, $\supp ( Y_{[p]} / N )$ contains only finitely many vertices $b$ and $b^+$ is finite.
  Then $Y_{[p]} / N$ is finitely presented by Corollary~\ref{cor:fd-fp}.
  Consider the exact sequence
  \[
    0 \To N \To Y_{[p]} \To Y_{[p]} / N \To 0.
  \]
  It follows that $Y_{[p]}$ is finitely presented.
\end{proof}

We show the injectivity of finitely presented $Y_{[p]}$ in $\fp(Q)$;
compare the dual of \cite[Proposition~6.2]{Jiao2019Projective}.

\begin{proposition}\label{prop:Yp-inj}
  Let $p$ be a left infinite path such that $Y_{[p]}$ is finitely presented. Then $Y_{[p]}$ is an indecomposable injective object in $\fp(Q)$.
\end{proposition}

\begin{proof}
  Assume
  \(
    p = \cdots \alpha_i \cdots \alpha_2 \alpha_1
  \).
  For each $i \geq 0$, we set $a_i = s(\alpha_{i+1})$.
  Consider the morphism
  \(
    \psi_{i+1} \colon I_{a_{i+1}} \to I_{a_i}
  \)
  given by
  \(
    (\psi_{i+1})_b (f) (u) = f(\alpha_{i+1} u)
  \)
  for any $f \in I_{a_{i+1}} (b)$ and $u \in Q(b, a_i)$.
  We observe that $(I_{a_i})_{i \geq 0}$ forms an inverse system, and $Y_{[p]}$ is the inverse limit in $\Rep(Q)$;
  see also \cite[Lemma~5.7]{Jiao2019Projective}.

  Given any exact sequence
  \[
    0 \To L \To M \To N \To 0
  \]
  in $\fp(Q)$, it is also an exact sequence in $\Rep(Q)$.
  Applying $\Hom(-, I_{a_i})$, we obtain an exact sequence of inverse systems of $k$-linear spaces
  \[
    0 \To ( \Hom(N, I_{a_i}) ) \To ( \Hom(M, I_{a_i}) ) \To ( \Hom(L, I_{a_i}) ) \To 0.
  \]
  Lemma~\ref{lem:proj-inj} implies that
  \(
    \Hom(N, I_{a_i}) \simeq \Hom_k( N(a_i), k )
  \).
  We observe by Lemma~\ref{lem:fp-pf} that it is finite dimensional.
  Then $( \Hom(N, I_{a_i}) )$ satisfies the Mittag-Leffler condition naturally.
  It follows from \cite[Proposition~3.5.7]{Weibel1994introduction} the exact sequence
  \[
    0 \To \varprojlim \Hom(N, I_{a_i}) \To \varprojlim \Hom(M, I_{a_i})
    \To \varprojlim \Hom(L, I_{a_i}) \To 0.
  \]
  For any $X \in \fp(Q)$, there exist natural isomorphisms
  \[
    \varprojlim \Hom(X, I_{a_i}) \simeq \Hom(X, \varprojlim I_{a_i})
    \simeq \Hom(X, Y_{[p]}).
  \]
  Then we obtain the exact sequence
  \[
    0 \To \Hom(N, Y_{[p]}) \To \Hom(M, Y_{[p]}) \To \Hom(L, Y_{[p]}) \To 0.
  \]
  It follows that $Y_{[p]}$ is an indecomposable injective object in $\fp(Q)$.
\end{proof}

Now, we can classify the indecomposable injective objects in $\fp(Q)$.

\begin{theorem}\label{thm:classify}
  Let $Q$ be an interval finite quiver. Assume $I$ is an indecomposable injective object in $\fp(Q)$.
  Then either $I \simeq I_a$ where $a$ admits only finitely many predecessors $b$ and each $b^+$ is finite, or $I \simeq Y_{[p]}$ where $[p]$ is an equivalence class of left infinite paths such that $\supp Y_{[p]}$ is top finite uniformly interval finite and
  \(
    \bigcup_{a \in \supp Y_{[p]}} a^+ \setminus \supp Y_{[p]}
  \)
  is finite.
\end{theorem}


\begin{proof}
  If $\supp I$ contains no left infinite paths, Proposition~\ref{prop:inj(Q)} implies that $I \simeq I_a$ for some vertex $a$.
  Corollary~\ref{cor:Ia-fp} implies that $a$ admits only finitely many predecessors $b$ and each $b^+$ is finite.

  Now, we assume $\supp I$ contains some left infinite paths.
  Let $\Delta$ be the set of left infinite paths $p$ contained in $\supp I$ with $s(p)^- \cap \supp I = \emptyset$. It follows from Lemma~\ref{lem:inj-supp-p-finite} that $\Delta$ is finite.

  For every $p \in \Delta$, we assume
  \(
    p = \cdots \alpha_{p,j} \cdots \alpha_{p,2} \alpha_{p,1}
  \).
  Set $a_{p,j} = s(\alpha_{p,j+1})$ for each $j \geq 0$.
  By Lemma~\ref{lem:inj-supp-p}, there exists some nonnegative integer $Z_p$ such that $a_{p,j}^+ = \set{a_{p,j+1}}$, $a_{p,j+1}^- = \set{a_{p,j}}$ and $I(\alpha_{p,j+1})$ is a bijection for any $j \geq Z_p$.

  Consider the subrepresentation $M$ of $I$ generated by $I(a_{p, Z_p+1})$ for all $p \in \Delta$. It follows from Lemma~\ref{lem:fp-M/N} that $I/M$ is finitely presented. It is an injective object in $\fp(Q)$, since $\fp(Q)$ is hereditary.

  We observe by Lemma~\ref{lem:inj-supp-p-finite} that $\supp (I/M)$ contains no left infinite paths.
  Indeed, assume
  \(
    \cdots \alpha_i \cdots \alpha_2 \alpha_1
  \)
  is a left infinite path in $\supp (I/M)$. It also lies in $\supp I$. Then these $\alpha_i$ for $i$ large enough lie in $\supp M$. Therefore, they do not lie in $\supp (I/M)$, which is a contradiction.

  It follows from Proposition~\ref{prop:inj(Q)} that
  \[
    I/M \simeq \bigoplus_{b \in Q_0} I_b^{\oplus \dim ( \soc (I/M) ) (b)}.
  \]
  For any $p \in \Delta$, we observe that
  \(
    ( \soc (I/M) ) (a_{p, Z_p}) = I(a_{p, Z_p}) \neq 0
  \)
  and $I(\alpha_{p,j+1})$ is a bijection for any $j \geq Z_p$.
  The previous isomorphism can be extended as
  \[
    I \simeq
    \biggl(
      \bigoplus_{b \in Q_0} I_b^{\oplus \dim (\soc I) (b)}
    \biggr)
    \oplus
    \biggl(
      \bigoplus_{p \in \Delta} Y_{[p]}^{\oplus \dim I(a_{p, Z_p})}
    \biggr).
  \]

  Since $I$ is indecomposable, then $\Delta$ contains only one left infinite path $p$ and $\dim I(a_{p, Z_p}) = 1$. Then $\soc I = 0$ and $I \simeq Y_{[p]}$.
  It follows from Lemma~\ref{lem:Yp-fp} that $\supp Y_{[p]}$ is top finite and uniformly interval finite, and
  \(
    \bigcup_{a \in \supp Y_{[p]}} a^+ \setminus \supp Y_{[p]}
  \)
  is finite.
\end{proof}

\begin{remark}\label{rmk:spectroid}
  Let $\mathcal{C}$ be a $k$-linear \emph{spectroid}, i.e., a Hom-finite category whose objects are pairwise non-isomorphic with local endomorphism rings.
  Assume $k$ is algebraically closed, and the infinite radical of $\mathcal{C}$ vanishes, and the category of modules over $\mathcal{C}$ is hereditary.
  Then the quiver of $\mathcal{C}$ is interval finite. It can be viewed as a category naturally, and its $k$-linearization is precisely $\mathcal{C}$;
  see \cite[Sections~8.1 and 8.2]{GabrielRoiter1992Representations} for more details.
  Therefore, Theorem~\ref{thm:classify} can be applied to the category of finitely presented modules over $\mathcal{C}$.
\end{remark}

\section{Examples}

Let $k$ be a field. We will give some examples.

\begin{example}
  Assume $Q$ is the following quiver.
  \[\begin{tikzcd}
    \cdots
    & \underset3\circ \lar["\alpha_4"']
    & \underset2\circ \lar["\alpha_3"']
    & \underset1\circ \lar["\alpha_2"']
    & \underset0\circ \lar["\alpha_1"']
  \end{tikzcd}\]

  For each $n \geq 0$, we consider the representation $I_n$.
  We observe that the predecessors of $n$ are $i$ for $0 \leq i \leq n$, and $i^+$ is finite.
  Corollary~\ref{cor:Ia-fp} implies that $I_n$ is finitely presented.

  Let
  \(
    p = \cdots \alpha_i \cdots \alpha_2 \alpha_1
  \).
  Then $\supp Y_{[p]} = Q$. We observe that $\supp Y_{[p]}$ is top finite uniformly interval finite and
  \(
    \bigcup_{a \in Q_0} a^+ \setminus Q_0
  \)
  is the empty set. Lemma~\ref{lem:Yp-fp} implies that $Y_{[p]}$ is finitely presented.

  We observe by Theorem~\ref{thm:classify} that
  \[
    \set{I_n | n \geq 0} \cup \set{Y_{[p]}}
  \]
  is a complete set of indecomposable injective objects in $\fp(Q)$.
\end{example}

\begin{example}
  Assume $Q$ is the following quiver.
  \[\begin{tikzcd}
    \cdots
    & \underset1\circ \lar["\alpha_2"']
    & \underset0\circ \lar["\alpha_1"']
    & \underset{-1}\circ \lar["\alpha_0"']
    & \cdots \lar["\alpha_{-1}"']
  \end{tikzcd}\]

  For each integer $n$, we consider the representation $I_n$.
  We observe that all $i \leq n$ are predecessors of $n$.
  Corollary~\ref{cor:Ia-fp} implies that $I_n$ is not finitely presented.

  Let
  \(
    p = \cdots \alpha_i \cdots \alpha_2 \alpha_1
  \).
  Then $\supp Y_{[p]} = Q$, which contains a right infinite path
  \(
    \alpha_{-1} \alpha_{-2} \cdots \alpha_{-i} \cdots
  \).
  Then it is not top finite by Lemma~\ref{lem:inf-path}. It follows from Lemma~\ref{lem:Yp-fp} that $Y_{[p]}$ is not finitely presented.

  We observe by Theorem~\ref{thm:classify} that $\fp(Q)$ contains no nonzero injective objects.
\end{example}

\begin{example}
  Assume $Q$ is the following quiver.
  \[\begin{tikzcd}[sep = 2.5em]
    &&&&& \overset0\circ \\
    \cdots & \underset{i}\circ & \cdots & \underset3\circ & \underset2\circ & \underset1\circ
    \ar[from = 1-6, to = 2-1, out = 180, in = 20, phantom, "\cdots" pos = 0.8]
    \ar[from = 1-6, to = 2-2, out = 180, in = 30, "\alpha_i"' near end]
    \ar[from = 1-6, to = 2-3, out = 180, in = 30, phantom, "\cdots" pos = 0.8]
    \ar[from = 1-6, to = 2-4, out = 190, in = 45, "\alpha_3" pos = 0.63]
    \ar[from = 1-6, to = 2-5, out = 205, in = 60, "\alpha_2"]
    \ar[from = 1-6, to = 2-6, "\alpha_1"]
    \ar[from = 2-2, to = 2-1, "\beta_{i+1}"]
    \ar[from = 2-3, to = 2-2, "\beta_i"]
    \ar[from = 2-4, to = 2-3, "\beta_4"]
    \ar[from = 2-5, to = 2-4, "\beta_3"]
    \ar[from = 2-6, to = 2-5, "\beta_2"]
  \end{tikzcd}\]
  We mention that $Q$ is interval finite, but not \emph{locally finite} (i.e., for any vertex $a$, the set of arrows $\alpha$ with $s(\alpha) = a$ or $t(\alpha) = a$ is finite).

  For each $n \geq 0$, we consider the representation $I_n$.
  We observe that the set of predecessors of $n$ is $\set{0 \leq i \leq n}$, but $0^+$ is not finite.
  Then Corollary~\ref{cor:Ia-fp} implies that $I_n$ is not finitely presented.

  Let
  \(
    p = \cdots \beta_i \cdots \beta_3 \beta_2
  \).
  Then $\supp Y_{[p]} = Q$, which is not uniformly interval finite. Lemma~\ref{lem:Yp-fp} implies that $Y_{[p]}$ is not finitely presented.

  We observe by Theorem~\ref{thm:classify} that $\fp(Q)$ contains no nonzero injective objects.
\end{example}


\begin{example}
  Assume $Q$ is the following quiver.
  \[\begin{tikzcd}
    \cdots
    & \overset{a_i}{\circ} \dar["\gamma_i"'] \lar["\alpha_{i+1}"']
    & \cdots \lar["\alpha_i"']
    & \overset{a_2}\circ \dar["\gamma_2"'] \lar["\alpha_3"']
    & \overset{a_1}\circ \dar["\gamma_1"'] \lar["\alpha_2"']
    & \overset{a_0}\circ \dar["\gamma_0"'] \lar["\alpha_1"']
    \\
    \cdots
    & \underset{b_i}{\circ} \lar["\beta_{i+1}"]
    & \cdots \lar["\beta_i"]
    & \underset{b_2}\circ \lar["\beta_3"]
    & \underset{b_1}\circ \lar["\beta_2"]
    & \underset{b_0}\circ \lar["\beta_1"]
  \end{tikzcd}\]

  For each $n \geq 0$, we consider the representations $I_{a_n}$ and $I_{b_n}$.
  We observe that the set of predecessors of $a_n$ is
  \(
    \set{a_i | 0 \leq i \leq n}
  \),
  and the one of $b_n$ is
  \(
    \set{a_i | 0 \leq i \leq n} \cup \set{b_i | 0 \leq i \leq n}
  \).
  Since each $a_i^+$ and $b_i^+$ are both finite, Corollary~\ref{cor:Ia-fp} implies that $I_{a_n}$ and $I_{b_n}$ are finitely presented.

  Let
  \(
    p = \cdots \alpha_i \cdots \alpha_2 \alpha_1
  \).
  Then $\supp Y_{[p]}$ is the full subquiver of $Q$ formed by $a_i$ for all $i \geq 0$.
  We observe that
  \(
    \bigcup_{a \in \supp Y_{[p]}} a^+ \setminus \supp Y_{[p]}
  \)
  contains all $b_i$ and then is infinite.
  Lemma~\ref{lem:Yp-fp} implies that $Y_{[p]}$ is not finitely presented.

  Let
  \(
    q = \cdots \beta_i \cdots \beta_2 \beta_1
  \).
  Then $\supp Y_{[q]} = Q$. Since $Q$ is not uniformly interval finite, then $Y_{[q]}$ is not finitely presented by Lemma~\ref{lem:Yp-fp}.

  We observe that $\set{[p], [q]}$ is the set of equivalence classes of left infinite paths.
  It follows from Theorem~\ref{thm:classify} that
  \[
    \set{I_{a_i} | i \geq 0} \cup \set{I_{b_i} | i \geq 0}
  \]
  is a complete set of indecomposable injective objects in $\fp(Q)$.
%
%
\end{example}

\begin{example}
  Assume $Q$ is the following quiver.
  \[\begin{tikzcd}
    \cdots
    & \overset{a_2}\circ \lar["\alpha_3"']
    & \overset{a_1}\circ \lar["\alpha_2"'] \dar["\gamma_1"']
    & \overset{a_0}\circ \lar["\alpha_1"'] \dar["\gamma_0"]
    \\
    \cdots
    & \underset{b_2}\circ \lar["\beta_3"]
    & \underset{b_1}\circ \lar["\beta_2"]
    & \underset{b_0}\circ \lar["\beta_1"]
    & \underset{b_{-1}}\circ \lar["\beta_0"]
    & \cdots \lar["\beta_{-1}"]
  \end{tikzcd}\]

  We observe that $a^+$ is finite for any $a \in Q_0$.
  Consider the representations $I_{a_i}$ for $i \geq 0$ and $I_{b_j}$ for $j \in \mathbb{Z}$.
  The set of predecessors of $a_i$ is finite and the one of $b_j$ is not. It follows from Corollary~\ref{cor:Ia-fp} that $I_{a_i}$ is finitely presented, while $I_{b_j}$ is not.

  Let
  \(
    p = \cdots \alpha_i \cdots \alpha_2 \alpha_1
  \).
  Then $\supp Y_{[p]}$ is the full subquiver of $Q$ formed by $a_i$ for all $i \geq 0$.
  We observe that $\supp Y_{[p]}$ is top finite uniformly interval finite, and
  \(
    \bigcup_{a \in \supp Y_{[p]}} a^+ \setminus \supp Y_{[p]} = \set{b_0, b_1}
  \).
  It follows from Lemma~\ref{lem:Yp-fp} that $Y_{[p]}$ is finitely presented.

  Let
  \(
    q = \cdots \beta_i \cdots \beta_2 \beta_1
  \).
  Then $\supp Y_{[q]}$ is the full subquiver of $Q$ formed by $a_0$, $a_1$ and $b_j$ for all $j \in \mathbb{Z}$.
  We observe that $\supp Y_{[q]}$ contains a right infinite path
  \(
    \beta_{-1} \beta_{-2} \cdots \beta_{-i} \cdots
  \).
  Then it is not top finite by Lemma~\ref{lem:inf-path}. It follows from Lemma~\ref{lem:Yp-fp} that $Y_{[q]}$ is not finitely presented.

  We observe that $\set{[p], [q]}$ is the set of equivalence classes of left infinite paths.
  It follows from Theorem~\ref{thm:classify} that
  \[
    \set{I_{a_i} | i \geq 0} \cup \set{Y_{[p]}}
  \]
  is a complete set of indecomposable injective objects in $\fp(Q)$.
\end{example}

\section*{Acknowledgements}

The author is very grateful to the referee for many helpful suggestions and comments,
especially the spectroid point of view in Remark~\ref{rmk:spectroid} and the simplifications of some proofs.



\begin{thebibliography}{EEGR09}

\bibitem[Aus66]{Auslander1966Coherent}
  M.~Auslander,
  \emph{Coherent functors},
  Proc. {C}onf. {C}ategorical {A}lgebra ({L}a {J}olla, {C}alif., 1965),
  Springer, New York, 1966, pp.~189--231.

\bibitem[BL17]{BautistaLiu2017bounded}
  R.~Bautista and S.~Liu,
  \emph{The bounded derived categories of an algebra with radical squared zero},
  J. Algebra \textbf{482} (2017), 303--345.

\bibitem[BLP13]{BautistaLiuPaquette2013Representation}
  R.~Bautista, S.~Liu, and C.~Paquette,
  \emph{Representation theory of strongly locally finite quivers},
  Proc. Lond. Math. Soc. (3) \textbf{106} (2013), 97--162.

\bibitem[BG82]{BongartzGabriel1982Covering}
  K.~Bongartz and P.~Gabriel,
  \emph{Covering spaces in representation-theory},
  Invent. Math. \textbf{65} (1982), 331--378.

\bibitem[Che15]{Chen2015Irreducible}
  X.-W. Chen,
  \emph{Irreducible representations of {L}eavitt path algebras},
  Forum Math. \textbf{27} (2015), 549--574.

\bibitem[EEGR09]{EnochsEstradaGarciaRozas2009Injective}
  E.~Enochs, S.~Estrada, and J.~R. Garc\'{i}a~Rozas,
  \emph{Injective representations of infinite quivers. {A}pplications},
  Canad. J. Math. \textbf{61} (2009), 315--335.

\bibitem[Gab81]{Gabriel1981universal}
  P.~Gabriel,
  \emph{The universal cover of a representation-finite algebra},
  Representations of algebras ({P}uebla, 1980),
  Lecture Notes in Mathematics, vol. 903,
  Springer, Berlin-New York, 1981, pp.~68--105.

\bibitem[GR92]{GabrielRoiter1992Representations}
  P.~Gabriel and A.~V. Roiter,
  \emph{Representations of finite-dimensional algebras},
  Algebra {VIII},
  Encyclopaedia of Mathematical Sciences, vol.~73,
  Springer, Berlin, 1992, pp.~1--177.

\bibitem[Jia19]{Jiao2019Projective}
  P.~Jiao,
  \emph{Projective objects in the category of pointwise finite
  dimensional representations of an interval finite quiver},
  Forum Math. \textbf{31} (2019), 1331--1349.

\bibitem[RVdB02]{ReitenVandenBergh2002Noetherian}
  I.~Reiten and M.~Van~den Bergh,
  \emph{Noetherian hereditary abelian categories satisfying {S}erre duality},
  J. Amer. Math. Soc. \textbf{15} (2002), 295--366.

\bibitem[Wei94]{Weibel1994introduction}
  C.~A. Weibel,
  \emph{An introduction to homological algebra},
  Cambridge Studies in Advanced Mathematics, vol.~38,
  Cambridge University Press, Cambridge, 1994.

\end{thebibliography}

\end{document}